\documentclass{birkjour}

\usepackage{amssymb}

\usepackage[ansinew]{inputenc}
\usepackage{euscript}

\usepackage{amsmath}
\usepackage{amsfonts}
\usepackage{amssymb}
\usepackage{amsthm}
\usepackage{euscript}
\usepackage{cite}
\usepackage{color}
\usepackage{enumerate} 
\usepackage[ansinew]{inputenc}
\usepackage{xcolor}
\usepackage[makeroom]{cancel}
\definecolor{mypink1}{rgb}{1.00,0.10,1.00}
\usepackage{ulem}
\usepackage{soul}

\newtheorem{theo}{Theorem}[section]
\newtheorem{lemma}[theo]{Lemma}
\newtheorem{cor}[theo]{Corollary}
\newtheorem{prop}[theo]{Proposition}

\newtheorem{defi}[theo]{Definition}
\newtheorem{remark}[theo]{Remark}
\newtheorem{example}[theo]{Example}

\definecolor{marron}{rgb}{0.1,0.4,0.1}

\def\R{{\rm I}\hskip-.13em{\rm R}}
\def\cC{{{\mathcal C}}}
\def\cK{{{\mathcal K}}}

\def\cA{{{\mathcal A}}}

\def\cero{0_X}
\begin{document}

\title[Separation Theorems of Co-radiant Sets and Optimality Conditions]{Nonlinear Separation Theorems for Co-Radiant Sets and Optimality Conditions for Approximate and Proper Approximate Solutions in Vector Optimization}

%----------Author 1
\author{Fernando Garc\'ia-Casta\~no}
\address{University de Alicante, Department of Mathematics, Carretera San Vicente del Raspeig, s/n, 03690 San Vicente del Raspeig, Alicante, Spain.\\ https://orcid.org/0000-0002-8352-8235 }
\email{fernando.gc@ua.es}

%\thanks{Orcid Garc\'ia-Casta\~no: }
%----------Author 2

\author{Miguel \'Angel Melguizo-Padial}

\address{%
University de Alicante, Deparment of Mathematics, Carretera San Vicente del Raspeig, s/n, 03690 San Vicente del Raspeig, Alicante, Spain.\\ https://orcid.org/0000-0003-0303-791X}

\email{ma.mp@ua.es}
%----------classification, keywords, date


\date{}
%----------additions
%\dedicatory{To my boss}
%%% ----------------------------------------------------------------------

\begin{abstract}
This paper deals with \(\varepsilon\)-efficient and \(\varepsilon\)-properly efficient points with respect to a co-radiant set in vector optimization problems. In the first part of the paper, we establish a new nonlinear separation theorem for co-radiant sets in normed spaces. Subsequently, we obtain necessary and sufficient conditions, via scalarization, for both \(\varepsilon\)-efficient and \(\varepsilon\)-properly efficient points in a general framework, without requiring any assumptions on the co-radiant set or convexity conditions on the sets under consideration. Consequently, our results are applicable in a broader range of settings than those previously addressed in the literature.
\end{abstract}

%%% ----------------------------------------------------------------------
\maketitle
%%% ----------------------------------------------------------------------
\section{Introduction}\label{intro}
In optimization, many types of solutions have been studied, each of which provides information in a particular sense.  We find, for example, the notions of strong solution, weak solution, proper solution, etc.,  (see \cite{Jahn2004,GarciaMelguizo2023} and references therein). Among them, we can find the concept of approximate solution that naturally arises, for instance, when simplifying practical problems or when applying certain methods to solve optimization programs. In these cases,  solutions close to the exact solutions of the problem may be proposed and leads to the concept of an approximate solution.
It is commonly accepted in the literature that the first notion of approximate solution introduced in vector optimization was that of Kutateladze in \cite{Kutateladze1979}.  This notion has been applied in various optimization topics, as can be seen, for example, in \cite{JiHong2008,Dutta2001,Gutierrez2004,Gutierrez2005,GUTIERREZ2005b,Idrissi1988,Isac1996,Liu1991,Liu1996,Liu1999,Loridan1984,Loridan1992,Ruhe1990,Tammer1992,Valyi1987,WHITE1998}. Subsequently, other notions of approximate solutions were introduced (see, for example, \cite{Helbig1994,Nemeth1986,Tanaka1994,Valyi1985,White1986}). Finally, the authors in \cite{Gutierrez2006a}  introduced a new notion of approximate solution called the \(\varepsilon\)-efficient point, based on the notion of co-radiant set, which provided a common framework for simultaneously studying all the different notions of approximate solutions that had appeared until then.  In this context, the authors of \cite{Gutierrez2006a} derived necessary and sufficient conditions via nonlinear scalarization. Their approach makes use of a functional obtained from a nonconvex separation theorem. These conditions required certain additional assumptions on the co-radiant sets which are involved in the definition of \(\varepsilon\)-efficient point. For example, for necessary conditions, it was required that the co-radiant set be star-shaped with a solid kernel and, for sufficient conditions,  it was necessary  to use an extra strictly local monotone function.

In this work, we focus on the study of $\varepsilon$-efficient points,  $\varepsilon$-properly efficient points, and the derivation of necessary and sufficient conditions for them through the use of sublinear scalarizations.  In our approach, we eliminate the aforementioned restrictions on the co-radiant sets found in the statements in \cite{Gutierrez2006a}. However, to achieve this, we need a functional to separate co-radiant sets.  The only theorem we have found in the literature dealing with the separation of co-radiant sets is \cite[Theorem 1]{Sayadibander2017}, which is not sufficiently general for our needs. This theorem only applies to closed co-radiant sets with a bounded base in finite-dimensional spaces. Moreover, it presents some unsolved technical difficulties.  For such reasons, we begin the work by establishing some new nonlinear separation results for co-radiant sets, as general as possible.  In the proof of such results, we make use of the main result of \cite{GarciaMelguizo2023}, using a level surface constructed from the Bishop-Phelps cone which can be interpolated between the corresponding cones generated by the co-radiant sets we want to separate.  %It is worth pointing out that our results in \cite{GarciaMelguizo2023} are related to those in \cite{Tammer2024} as shown by the authors in their paper. 
After that, in the last part of this work, we determine necessary and sufficient conditions for \(\varepsilon\)-efficient points and \(\varepsilon\)-properly efficient points via sublinear scalarization.  We consider co-radiant sets of Bishop-Phelps type and study the $\varepsilon$-efficient points associated with them. We prove that those points are solutions to a sublinear scalar optimization problem and we establish  connections between them and the \(\varepsilon\)-efficient points associated with general co-radiant sets.  As a consequence of this analysis, we obtain our necessary and sufficient conditions which do not need the restrictions found in the results of \cite{Gutierrez2006a} as well as those in \cite{Sayadibander2017}.  

The plan of the paper is as follows.  In the subsequent section, we introduce the basic notation and terminology.   In Section \ref{SeparationGeneral}, we introduce the separation property for cones we need from \cite{GarciaMelguizo2023} and we establish some results regarding such a property related to the cones generated by co-radiant sets.  The results in this section are preparatory for Section \ref{Subsection:Proofs_Separation_Theorems} in which we establish our theorems of separation of co-radiant sets.  In Section \ref{Scalarization}, we introduce the notions of $\varepsilon$-efficient and $\varepsilon$-properly efficient points and relate them to scalarization techniques. 

\section{Notation and Previous Terminology} \label{Preeliminares}
Throughout the paper, let \( X \) denote a normed space, with \( \| \cdot \| \) denoting its norm, \( X^* \) the dual Banach space of \( X \), and \( \| \cdot \|_* \) the corresponding dual norm. We denote by \( B_X := \{ x \in X : \|x\| \leq 1 \} \) (resp. \( B_{X^*} := \{ f \in X^* : \|f\|_* \leq 1 \} \)) the closed unit ball of \( X \) (resp. of \( X^* \)), and by \( S_X := \{ x \in X : \|x\| = 1 \} \) (resp. \( S_{X^*} := \{ f \in X^* : \|f\|_* = 1 \} \)) the unit sphere of \( X \) (resp. of \( X^* \)). Additionally, we define \( B_X(x, \varepsilon) := \{ x' \in X : \| x - x' \| \leq \varepsilon \} \).

Given a subset \( S \subset X \), we denote by \( \overline{S} \) (resp. \( \text{bd}(S) \), \( \text{int}(S) \), \( \text{co}(S) \), \( \overline{\text{co}}(S) \), \( S^c \)) the closure (resp. boundary, interior, convex hull, closure of the convex hull, complement) of \( S \). We say that a set \( S \) is {solid} if \( \text{int}(S) \neq \emptyset \), {proper} if \( S \neq \emptyset \) and \( S \neq X \), and {pointed} if \( S \cap (-S) \subset \{ \cero \} \), i.e., if \( S \cap (-S) = \{ \cero \} \) when \( \cero \in S \), and if \( S \cap (-S) = \emptyset \) when \( \cero \notin S \).

A subset $C \subset X$ is said to be a cone if $\lambda x \in C$ for every $\lambda \geq 0$ and $x \in C$.  A cone $C$ is said to be  a convex cone if $C$ is a convex subset of $X$. Fixed a subset $S \subset X$, we define the cone generated by $S$ as cone$(S):=\{\lambda s \colon \lambda\geq 0,\, s \in S\}$ and $\overline{\mbox{cone}}(S)$ stands for the closure of cone$(S)$.  A non-empty convex subset $B$ of a convex cone $C$ is said to be a base of $C$  if $\cero \not \in \overline{B}$ and for every $x\in C$, $x\not =\cero$, there exists unique $\lambda_x>0$ and $b_x\in B$ such that $x=\lambda_xb_x$.  A convex cone $C$ is said to have a bounded base if there exists a base $B$ of $C$ such that it is a bounded subset of $X$. Given a cone $C\subset X$, it is defined the dual cone of $C$ by $C^*:=\{f \in X^*\colon f(x)\geq 0,\, \forall x \in C\}$ and the quasi-interior of $C^*$ by $C^{\#}:=\{f \in X^*\colon f(x)> 0,\, \forall x \in C \setminus \{0_X\} \}$.   It is known that $C^{\#}\supset \mbox{int}(C^*)$ and that any   convex cone $C\subset X$ has a base if and only if $C^{\#}\not = \emptyset$. 

The following sets were introduced in \cite{GarciaMelguizo2023}  and they are called augmented dual cones, $C^{a*}_+:=\{(f,\alpha)\in C^{\#}\times (0,+\infty) \colon f(x)-\alpha\|x \| \geq 0,\, \forall x \in C\}$ and $C^{a\#}_+:=\{(f,\alpha)\in C^{\#}\times (0,+\infty) \colon f(x)-\alpha\| x \| > 0,\, \forall x \in C \setminus \{ 0_X\}\}$. A function $p:X\rightarrow \R$ is said to be sublinear if it is both positively homogeneous and subadditive, i.e.,  $p(\alpha x)=\alpha p(x)$, for every $\alpha\geq 0$, and $p(x+y)\leq p(x)+p(y)$, for all $x$, $y \in X$, respectively.

We say that a subset $\cC \subset X$ is co-radiant if  $\lambda x \in \cC$ for every $\lambda  \geq  1$ and  $x \in \cC$.  It is clear that every cone in a normed space is also a co-radiant set.   All co-radiant sets in this manuscript are assumed to be proper.  We define the  dual co-radiant set of $\cC$ as 
$\cC^*:=\{ f \in X^*: f(y) \geq 0, \,  \forall y \in \cC \}$  and the quasi-interior of $\cC^*$ as $\cC^{\#}:=\{f \in X^*: f(y) > 0,\,  \forall  y \in \cC \}$.  Let $0<\varepsilon \leq 1$ and $\cC \subset X $ be a co-radiant set,  we define   $\cC(\varepsilon):=\varepsilon \cC$. Now,  fix $\varepsilon,\,\lambda> 0$. Following \cite{Sayadibander2017}, we define the augmented $\varepsilon$-dual co-radiant set of $\cC$ by
\begin{equation*}
\cC^{a*}_{\lambda}(\varepsilon):=\left\lbrace (f,\alpha) \in \cC^{\#} \times   \mathbb{R}_+: f(y) -\alpha \parallel y \parallel \geq \lambda,  \,  \forall y \in \cC(\varepsilon) \right\rbrace, \\
\end{equation*}
and the augmented  $\varepsilon$-quasi-interior of $\cC^{a*}_{\lambda}(\varepsilon)$ by
\begin{equation*}
\cC^{a\#}_{\lambda}(\varepsilon):=\left\lbrace (f,\alpha) \in\cC^{\#} \times   \mathbb{R}_+: f(y) -\alpha \parallel y \parallel > \lambda,  \,  \forall y \in \cC(\varepsilon) \right\rbrace.
\end{equation*}

\section{Some Preliminary Results Regarding Cones and Co-Radiant Sets} \label{SeparationGeneral}
We begin this section with a separation property for cones introduced in \cite{GarciaMelguizo2023}. In the first part of the section, we study some stability properties of this separation property when one of the cones involved is dilated. The second part of the section is devoted to analysing the properties of the norm bases of co-radiant sets, as well as the properties of the cones generated by them. These results will be useful in the proofs presented in the following section.

\begin{defi}(\cite[Definition 2]{GarciaMelguizo2023})\label{SSPconos}
Let $X$ be a normed space and let $C$ and $K$ be cones in $X$. We say that the pair of cones $(C,K)$ satisfies the strict separation property (SSP for short) if $\cero\not \in \overline{\mbox{co}{(C\cap S_X)-\mbox{co}((bd(K)\cap S_X)\cup\{\cero\})}}$.
\end{defi}

We next state a characterization of the SSP, proved in \cite{GarciaMelguizo2023}, which will be used throughout this work.

\begin{theo}(\cite[Theorem 3.1] {GarciaMelguizo2023})\label{Tma1_sep_conos_MMOR}
Let $X$ be a normed space and let $C$ and $K$ be cones in $X$. Then the pair $(C,K)$ satisfies the SSP if and only if there exist $\alpha_1$, $\alpha_2 \in \mathbb{R}$ and $f \in X^*$ such that $0 < \alpha_1 < \alpha_2$ and $f(x) + \alpha \|x\| < 0 < f(y) + \alpha \|y\|$, for every $\alpha\in (\alpha_1,\alpha_2)$, $x\in -C\setminus \{\cero\}$, and $y\in \mbox{bd}(-K)\setminus \{\cero\}$.
\end{theo}

%Next, we introduce (from \cite{GarciaMelguizo2024}) the concept of $\delta$-conic neighborhood of a cone. 
%\begin{defi}
%Let  $X$ be a normed space, let $C \subset X$ be a cone, $0<\delta<1$, and consider $S_{\delta}:=\{y \in X \colon d(y, C\cap S_X)\leq \delta\}$. The $\delta$-conic neighbourhood of $C$ is the cone defined by $C_{\delta}:=\mbox{cone}(S_{\delta})$. 
%\end{defi}

Next, we introduce (from \cite{GarciaMelguizo2024}) the concept of a $\delta$-conic neighborhood of a cone. 

\begin{defi}
Let $X$ be a normed space, let $C \subset X$ be a cone, and let $0<\delta<1$. 
For any $y \in X$ and set $A \subset X$, let \(d(y,A) := \inf_{a \in A} \|y-a\|\)
denote the distance from $y$ to $A$.  
Consider \(S_{\delta} := \{y \in X \colon d(y, C \cap S_X) \leq \delta\}\).
The $\delta$-conic neighbourhood of $C$ is the cone defined by \(C_{\delta} := \mathrm{cone}(S_{\delta})\).
\end{defi}

The following result establishes that, for $\delta$ small enough, the SSP also keeps holding when we dilate $C$ making use of a $\delta$-conic neighbourhood.

\begin{lemma}\label{lema:cono_ampliado_SSP}
Let $X$ be a normed space and let $C$ and $K$ be cones in $X$. If the pair $(C,K)$ satisfies the SSP, then there exists $\delta' >0$ such that the pair $(C_{\delta},K)$ satisfies the SSP for every $\delta \in (0, \delta')$.
\end{lemma}
\begin{proof}
Assume that the pair $(C,K)$ satisfies the SSP. By Theorem \ref{Tma1_sep_conos_MMOR}, there exist $\delta_1, \delta_2 \in \mathbb{R}$ and $f \in C^{\#}$ such that $0 < \delta_1 < \delta_2$, and for every $\alpha \in (\delta_1, \delta_2)$, we have $(f, \alpha) \in C^{a\#}_+$ satisfying $f(x) + \alpha \|x\| < 0 < f(y) + \alpha \|y\|$ for all $x \in -C \setminus \{\cero\}$ and $y \in \mathrm{bd}(-K) \setminus \{\cero\}$. Now pick $\alpha_0, \alpha_1, \alpha_2 \in \mathbb{R}$ such that $0<\delta_1<\alpha_1<\alpha_0<\alpha_2 <\delta_2$ and consider $\delta' :=\frac{\alpha_2-\alpha_0}{2(\| f \|+ \alpha_0)}$. Let  $\delta \in (0, \delta') $ and fix $x\in -\mbox{cone}(S_X \cap C+ 2\delta B_X) \setminus \lbrace 0_X\rbrace$. Then  $f(x)+ \alpha \| x \| <0$ for every $\alpha \in [\alpha_1,\alpha_0]$. Indeed, we write $x=-\lambda(s+2\delta b)$ for some $s \in C \cap S_X$, $b \in B_X$, and $\lambda>0$.  Thus 
\begin{align*}
f(x) + \alpha \|x\| 
&\leq f(x) + \alpha_0 \|x\| \\
&= f(-\lambda(s+2\delta b)) + \alpha_0 \|\lambda(s+2\delta b)\| \\
&= f(-\lambda s) + 2\delta \lambda f(-b) + \alpha_0 \|\lambda s + 2\lambda \delta b\| \\
&\leq f(-\lambda s) + 2\delta \lambda f(-b) + \alpha_0 \|\lambda s\| + 2\alpha_0 \lambda \delta \|b\| \\
&\leq f(-\lambda s) + 2\delta \lambda \|f\|_* \|b\| +  \alpha_0 \|\lambda s\| +2\alpha_0 \lambda \delta \|b\| \\
&\leq f(-\lambda s) + 2\delta \lambda \|f\|_* + \alpha_0 \|\lambda s\| + 2\alpha_0 \lambda \delta \\
&= f(-\lambda s) + \alpha_0 \|\lambda s\| + 2\delta \lambda(\|f\|_* + \alpha_0) \\
&\leq f(-\lambda s) + \alpha_0 \|\lambda s\| + \lambda(\alpha_2 - \alpha_0) \\
&= f(-\lambda s) + \alpha_0 \|\lambda s\| + (\alpha_2 - \alpha_0)\|\lambda s\| \\
&= f(-\lambda s) + \alpha_2 \|-\lambda s\| < 0,
\end{align*}
where the second equality follows by the linearity of $f \in X^*$, 
the second inequality by the triangle inequality, 
the third inequality by the relation $f(-b) \leq |f(-b)| \leq \|f\|^* \|b\|$, 
and the fourth inequality by the fact that $\|b\|\leq~1$.
%where the second equality follows by the linearity of $f\in X^*$,  the second inequality by the triangle inequality,  the third inequality by the inequalities$f(-b) \leq |f(-b)| \leq \|f\|_* \cdot \|b\|$, and the fourth inequality by $\|b\|\leq 1$.
%$f(x)+ \alpha \| x \|\leq f(x)+ \alpha_0 \| x \| = f(-\lambda(s+2\delta b))+ \alpha_0 \| \lambda(s+2\delta b) \|=$
%\[
%= f(-\lambda s) + 2\delta \lambda f(-b) + \alpha_0 \| \lambda s + 2\lambda \delta b \| 
%\quad \text{(by the linearity of $f \in X^*$)}
%\]
%\[
%\leq f(-\lambda s) + 2\delta \lambda f(-b) + \alpha_0 \| \lambda s \| + 2\alpha_0 \lambda \delta \| b \| 
%\quad \text{(by the triangle inequality)}
%\]
%\[
%\leq f(-\lambda s)+2\delta \lambda \|f\|_* \cdot \|b\| + \alpha_0 \| \lambda s\| + 2\alpha_0\lambda \delta \|b \| \quad \text{(since $f(-b) \leq |f(-b)| \leq \|f\|_* \cdot \|b\|$)}\]
%\[\leq f(-\lambda s) + \alpha_0 \| \lambda s \| + 2 \lambda \delta \| f \|_* +  2 \alpha_0 \lambda \delta 
%\quad \text{(since $\|b\| \leq 1$)}
%\]
%$=f(-\lambda s) + \alpha_0 \| \lambda s \| + 2 \lambda \delta (\| f \|_* + \alpha_0)\leq f(-\lambda s) + \alpha_0 \| \lambda s \| + \lambda  (\alpha_2 - \alpha_0)=f(-\lambda s) + \alpha_0 \| \lambda s \| +  (\alpha_2 - \alpha_0) \| \lambda s \|=f(-\lambda s) + \alpha_2 \| -\lambda s \| <0$.  
Now,  since  $C_{\delta}\ =\mbox{cone}(\overline{S_X \cap C+\delta B_X})\subset \mbox{cone}(S_X \cap C+2\delta B_X)$, it follows that for every $\alpha \in [\alpha_1,\alpha_0]$ we have $f(x)+\alpha \| x \| <0< f(y)+\alpha \| y \|$ whenever $x \in -C_{\delta}\setminus\{\cero\}$ and $y \in \mbox{bd}(-K)\setminus\lbrace \cero \rbrace$.  Furthermore,  the inequality   $f(x)+\alpha \| x \| <0$ for every $x \in -C_{\delta}\setminus\{\cero\}$ leads to $f \in (C_{\delta})^{a\#}_+$. Then, Theorem \ref{Tma1_sep_conos_MMOR} (ii) $ \Rightarrow$ (i) applies, concluding that the pair $(C_{\delta},K)$ satisfies the SSP.
\end{proof}

Below, we introduce the concept of norm base for a co-radiant set. This notion was introduced and used in \cite{Sayadibander2017} although not explicitly defined. Then, we will analyse some of its properties which are particularly helpful in our work as they facilitate a better understanding of the corresponding co-radiant sets. 
\color{black}
\begin{defi} 
Let $X$ be a normed space, let $\cC \subset X $ be a co-radiant subset, and $t>0$. We say that the set $\cC_t=\cC\cap t S_X$ is a norm base of $\cC$ with radius $t$ if for every $y \in \cC $ there exist $\lambda >0$ and $y_t \in \cC_t$ such that $y=\lambda \, y_t $.
\end{defi}

To illustrate this concept, we first present a simple example of a co-radiant set that admits a norm base. However, possessing a norm base is a rather stringent requirement for a co-radiant set. In Example~\ref{ejemplo}, we provide several co-radiant sets that do not admit such a base.

\begin{example} \label{ejemplo_con_normabase}
Given  $\cC:=\{ (x,y,z)\in \mathbb{R}^3: x,y,z\geq 0, x+y+z\geq 1 \}$,
it is easy to verify that for every $t \geq 1 $, the set
 $\cC_t:=\{(x,y,z)\in \mathbb{R}^3: x,y,z\geq 0, x^2+y^2+z^2=t^2 \}$ is a norm base of $\cC$.
\end{example}

Our following objective is to establish a characterization for the pair $(\mbox{cone}(\cC), \mbox{cone}(\cK))$ to satisfy the SSP in terms of the co-radiant sets $\cC$ and $\cK$. For such a purpose, we need the following  lemma.  In what follows, we write $I_{\cC}:= \inf \{ t>0:  \cC_t \mbox{ is a norm base of } \cC \}$ for any co-radiant set $\cC$ admitting a norm base.
\begin{lemma}\label{Lemma_antes_claims}
Let $X$ be a normed space, let $\cC\subset X$ be a co-radiant subset admitting a norm base, and $t>I_{\cC}$.  The following assertions hold.
\begin{itemize}
\item[(i)] $\mbox{cone}(\cC_t) =\mbox{cone}(\cC)$.
\item[(ii)] $\cC_t=\mbox{cone}(\cC_t) \cap tS_X$.
\item[(iii)] $\cC_t=t(\mbox{cone}(\cC)\cap S_X)$.  As a consequence,
if $\|x\|>I_{\cC}$, then $x\in \cC$ if and only if $x\in C$.
\item[(iv)] $\mbox{bd}(\cC) \cap t S_X= \mbox{bd}(\mbox{cone}(\cC)) \cap tS_X$.
\end{itemize}
\end{lemma}
\begin{proof}
To simplify notation, we will denote $C:=\mbox{cone}(\cC)$.

(i) It is clear, because $\cC_t$ is a norm base of $\cC$.

(ii)  The inclusion $\subset$ is clear. For the reverse, we fix $x\in \mbox{cone}(\cC_t)\cap tS_X$. Then, there exist $\lambda>0$ and $y\in \cC_t$ such that $x=\lambda y$. Since $t=\|x\|=\|\lambda y\|=\lambda t$, it follows that $\lambda=1$. Hence $x=y\in \cC_t$.

(iii) 
We first prove the identity. For the inclusion \( \subset \), let \( y \in \cC_t \). Then \( \|y\| = t \) and \( \frac{y}{\|y\|} \in \mathrm{cone}(\cC) \cap S_X \), so \( y = t \cdot \frac{y}{\|y\|} \in t(\mathrm{cone}(\cC) \cap S_X) \). For the reverse inclusion \( \supset \), let \( y \in t(C \cap S_X) \). Then \( y = t u \) for some \( u \in C \cap S_X = \mathrm{cone}(\cC) \cap S_X \). By items (i) and (ii), we have \( y \in \mathrm{cone}(\cC_t) \cap t S_X = \cC_t \), which completes the proof of the identity.

We now prove the final assertion. Fix \( x \in X \) such that \( \|x\| > I_{\cC} \). Since \( \cC \subset C \), we only need to prove that \( x \in C \) implies \( x \in \cC \). So assume \( x \in C \) with \( \|x\| = t > I_{\cC} \). Then \( \frac{x}{\|x\|} \in C \cap S_X \), and by the identity just proved, \( x = t \cdot \frac{x}{\|x\|} \in t(C \cap S_X) = \cC_t \subset \cC \), as desired.

(iv) Let us first prove the inclusion \( \mathrm{bd}(\cC) \cap t S_X \subset \mathrm{bd}(C) \cap t S_X \). Fix \( y \in \mathrm{bd}(\cC) \cap t S_X \), and take \( \varepsilon > 0 \) such that \( B_X(y, t\varepsilon) \cap B_X(0_X, I_{\cC}) = \emptyset \). Since \( y \in \mathrm{bd}(\cC) \), there exists \( y' \in \cC \cap B_X(y, t\varepsilon) \) such that \( \frac{y'}{t} \in C \) and \( \left\| \frac{y'}{t} - \frac{y}{t} \right\| < \varepsilon \). On the other hand, the condition \( B_X(y, t\varepsilon) \cap B_X(0_X, I_{\cC}) = \emptyset \) implies that every \( x'' \in \cC^c \cap B_X(y, t\varepsilon) \) satisfies \( \|x''\| > I_{\cC} \), and thus, by the last assertion in (iii), \( x'' \in C^c \). Therefore, there exists \( x'' \in C^c \cap B_X(y, t\varepsilon) \) such that \( \left\| \frac{x''}{t} - \frac{y}{t} \right\| < \varepsilon \) and \( \frac{x''}{t} \notin C \). This shows that \( \frac{y}{t} \in \mathrm{bd}(C) \cap S_X \), and hence
\[
y \in t(\mathrm{bd}(C) \cap S_X) = \mathrm{bd}(C) \cap t S_X,
\]
since \( t \cdot \mathrm{bd}(C) = \mathrm{bd}(C) \) for every cone \( C \subset X \) and every \( t > 0 \).

 Now we  prove the  inclusion $t(\mbox{bd}(C) \cap S_X)\subset \mbox{bd}(\cC) \cap t S_X$. Fix $y \in  \mbox{bd}(C) \cap S_X$, it is sufficient to check that $ty \in  \mbox{bd}(\cC)$.  Fix $\varepsilon >0$, there exists $y' \not\in C$ such that $\| y'-y \| <\frac{\varepsilon}{t}$. Then, $\| ty'-ty \|<\varepsilon$ and $ty' \not\in  C$. Now,  the inclusion   $ \cC \subset C$ yields $ty' \not\in \cC$. On the other hand, since $y \in \mbox{bd}(C)$, there exists $x'' \in C$ such that $\| x''-y \| <\frac{\varepsilon}{3t}$. Now, we consider the following cases.
\begin{itemize}
\item[(a)] Case $\| x'' \|=1$. Then, $tx'' \in C\cap tS_X=\mbox{cone}(\cC_t) \cap tS_X= \cC_t \subset \cC $.  Then, $tx'' \in \cC $ and $\|tx''-ty \| <\varepsilon$.
\item[(b)] Case $\| x'' \|<1$. Denote $u=\frac{x''}{\|x''\|}$. We will check first the inequality $\|tu-ty \| <\varepsilon$ and later that $tu \in \cC $.  Consider $v:=u-x''$. Then $\| v\|=1-\|x''\| \leq 1-(1-\frac{\varepsilon}{3t})=\frac{\varepsilon}{3t}$, because $\| x''\| \geq \|y\|-\|y-x''\|\geq 1-\frac{\varepsilon}{3t}$.  As a consequence,  $\|u-y\|=\|x''+v-y \| \leq \|y-x''\| +\|v\|\leq \frac{\varepsilon}{3t}+ \frac{\varepsilon}{3t}=\frac{2\varepsilon}{3t}<\frac{\varepsilon}{t}$.  Then $\|tu-ty\| <\varepsilon$.  For the last claim; since $u \in C\cap S_X $ and by (iii), it follows that $tu \in t(C \cap S_X)= \cC_t \subset \cC$.
\item[(c)] Case $\| x'' \|>1$. Denoting again $u=\frac{x''}{\|x''\|}$, the proof is analogous to the previous case but using $\| x''\| \leq 1+\frac{\varepsilon}{3t}$ and $\|u-x''\|=\|x''\|-1$. 
\end{itemize}
To sum up,  given arbitrary $y \in  \mbox{bd}(C) \cap S_X$ and $\varepsilon >0$, there exist $y'  \in \cC $  and $x''  \not\in \cC$ such that $\| y' -ty \|  <\varepsilon $ and $\| x'' -ty\|  <\varepsilon $. Then, $ty \in  \mbox{bd}(\cC) \cap t S_X$.  As a consequence, we have $\mbox{bd}(C) \cap tS_X=t(\mbox{bd}(C) \cap S_X)\subset \mbox{bd}(\cC) \cap t S_X$.
\end{proof} 

\begin{prop} \label{SSPcoradiantAmpliado}
Let $X$ be a normed space, and let $\cC$ and $\cK$ be co-radiant subsets of $X$. Assume that $\cK$ admits a norm base,  fix some $t>I_{\cK}$ and take $\cK_t=\cK \cap tS_X$. The following are equivalent.
\begin{itemize}
\item[(i)] The pair $(\mbox{cone}(\cC),\mbox{cone}(\cK_t))$ satisfies the SSP.
\item[(ii)] $0_X \not\in \overline{co(\mbox{cone}(\cC) \cap t S_X)-\mbox{co}((   bd(\mbox{cone}(\cK_t))\cap tS_X) \cup \{ 0_X \})}$.
\end{itemize}
\end{prop}
\begin{proof}
In order to simplify the notation, we will denote $C=\mbox{cone}(\cC)$ and $K_t=\mbox{cone}(\cK_t)$. Let us show (i)$\Leftrightarrow$(ii) directly. By definition, the pair $(C,K_t)$ satisfies the SSP if and only if
\[0_X \not \in \overline{\mbox{co}(C\cap S_X)-\mbox{co}((\mbox{bd}(K_t)\cap S_X)\cup\{0_X\}))}\Leftrightarrow\]
\[0_X \not \in t\overline{\mbox{co}(C\cap S_X)-\mbox{co}((\mbox{bd}(K_t)\cap S_X)\cup\{0_X\}))}, \mbox{ for any } t>0.\]
Now, it is not difficult to show that the former set coincides with the following one  $\overline{\mbox{co}(C \cap t S_X)-\mbox{co}((   bd(K_t)\cap tS_X) \cup \{ 0_X \})}$, and the proof finishes.
\end{proof}

When both co-radiant sets $\cC$ and $\cK$ admit a norm base,  the precedent result becomes the following one.  

\begin{prop} \label{teoremadelosSSP}
Let $X$ be a normed space, let $\cC$ and $\cK$  be co-radiant sets of $X$ admitting a norm base, and  $t>\max\{I_{\cC},I_{\cK}\}$.The following are equivalent.
\begin{itemize}
\item[(i)] The pair $(\mbox{cone}(\cC),\mbox{cone}(\cK))$ satisfies the SSP,
\item[(ii)] The pair $(\mbox{cone}(\cC_t),\mbox{cone}(\cK_t))$ satisfies the SSP,
\item[(iii)]  $0_X\not \in \overline{\mbox{co}(\cC_t)-\mbox{co}((\mbox{bd}(\cK)\cap tS_X) \cup \{ 0_X \})}$.
\end{itemize}
\end{prop}
\begin{proof}  
Proposition \ref{SSPcoradiantAmpliado} and assertion (i) in Lemma \ref{Lemma_antes_claims} give (i)$\Leftrightarrow$(ii).  For (ii)$\Leftrightarrow$(iii) we apply Proposition \ref{SSPcoradiantAmpliado} and assertions (i), (ii), and (iv) in Lemma~\ref{Lemma_antes_claims} 
\end{proof}

\section{Separation Theorems for Co-Radiant Sets} \label{Subsection:Proofs_Separation_Theorems}
In this section, we present several separation theorems for a couple of co-radiant sets $\mathcal{C}$ and $\mathcal{K}$.   In essence,  we establish that if the pair $(\mbox{cone}(\cC),\mbox{cone}(\cK))$ satisfies the SSP,  then there exists a co-radiant set of Bishop-Phelps type in the form of $C(f, \alpha, \lambda)=\{x\in X\colon f(x)- \alpha \|x\|\geq \lambda \}$ (for some $f\in X^*$, $\alpha>0$, and $\lambda >0$) separating the sets $\overline{\text{co}}(\mathcal{C})$ and $\mbox{bd}(\cK(\eta))$, for $\eta>0$ small enough. 
It is worth pointing out that the parameter $\eta>0$ represents a small positive parameter that allows the extension of the set $\cK$ towards the origin to enable the separation of  $\text{bd}(\cK(\eta))$ from $\overline{\text{co}}(\cC)$ by means of the hyperbolic surface $\{x\in X\colon f(x)- \alpha \| x \|=\lambda\}$. Furthermore, we establish that under a more restrictive assumption on the relative position between $\cC$ and $\cK$, it is possible to separate the sets $\overline{\text{co}}(\mathcal{C})$ and $X\setminus \text{int}(\mathcal{K}(\eta))$ by the same set $C(f, \alpha, \lambda)$.  From now on,  given a co-radiant set $\cC\subset X$, we will denote $d_{\cC}:=\inf \{ \| c\| : c \in \cC \}\geq0$. It is clear that $0\leq d_{\cC}\leq I_{\cC}$.

\begin{theo} \label{teoolarioseparacioncoradiantes2}
Let \( X \) be a normed space, and let \( \cC\) and \(\cK \) be co-radiant subsets of $X$ such that \( d_{\cC} d_{\cK} > 0 \) and \( \cK \) admits a norm base. Consider the following assertions:
\begin{itemize}
\item[(i)] The pair \( (\mathrm{cone}(\cC), \mathrm{cone}(\cK)) \) satisfies the SSP.
\item[(ii)] There exist \( \alpha_1, \alpha_2 \in \mathbb{R} \) with \( 0 < \alpha_1 < \alpha_2 \), and \( f \in S_{X^*} \) such that:
\begin{itemize}
\item[(a)] \( \lambda := \inf\{ f(y) - \alpha \| y \| : y \in \cC,\, \alpha \in (\alpha_1, \alpha_2) \} > 0 \),
\item[(b)] for every \( 0 < \eta < \frac{\lambda}{I_{\cK}} \), the inequalities
\begin{equation}\label{separacioncoradiant2_bis}
f(y') - \alpha \| y' \| < \lambda < f(y) - \alpha \| y \|,
\end{equation}
hold for all \( \alpha \in (\alpha_1, \alpha_2) \), \( y' \in \mathrm{bd}(\cK(\eta)) \), and \( y \in \overline{\mathrm{co}}(\cC) \).
\end{itemize}

\end{itemize}
Then, \( \text{(i)} \Rightarrow \text{(ii)} \). Moreover, if \( \cC \) also admits a norm base, then \( \text{(ii)} \Rightarrow \text{(i)} \).
\end{theo}

Next, we will establish some auxiliary results that will allow us to prove Theorem  \ref{teoolarioseparacioncoradiantes2}.  
\begin{prop}  \label{SSPepsiloncoradiant} 
Let $X$ be a normed space, let $\cK \subset X$ be a co-radiant subset, and fix $\eta>0$.  If $\cK$ admits a norm base, then $\cK(\eta)$ also has it. Furthermore, for any $t>0$ we have the equivalence $t>I_{\cK}\Leftrightarrow \eta t>I_{\cK(\eta)}$. As a consequence, we have the equality $I_{\cK(\eta)}=\eta I_{\cK}$.
\end{prop}
\begin{proof}
It is clear that $\cK(\eta)$ is co-radiant. Let us check that $\cK(\eta)$ admits a norm base. Fix $t>I_{\cK}$, and we will check that $\cK(\eta)_{t\eta}$ is a norm base of $\cK(\eta)$.  Fix arbitrary $y\in \cK(\eta)$. Then $y=\eta x$ for some $x \in \cK$, and there exist $z\in \cK\cap t S_X$ and $\lambda>0$ such that $x=\lambda z$. Thus, $y= \lambda \eta z$ and $\eta z\in \cK(\eta)\cap (t\eta S_X)=\cK(\eta)_{t\eta}$. So, $\cK(\eta)$ admits a norm base.  Note that our previous argument also proves the implication $t>I_{\cK}\Rightarrow t\eta>I_{\cK(\eta)}$.On the other hand,  it is straightforward to check that $\eta t>I_{\cK(\eta)}\Rightarrow$  $t>I_{\cK}$. As a consequence, $\cK_t$ is a norm base of $\cK$ if and only if $\cK(\eta)_{\eta t}$ is a norm base of $\cK(\eta)$. Therefore, 
\[
\begin{aligned}
\eta I_{\cK} &= \eta \inf\{t > 0 : \cK_t \text{ is a norm base of } \cK\} \\
&= \inf\{\eta t > 0 : \cK_t \text{ is a norm base of } \cK\} \\
&= \inf\{\eta t > 0 : \cK(\eta)_{\eta t} \text{ is a norm base of } \cK(\eta)\} \\
&= I_{\cK(\eta)}.
\end{aligned}
\]
\end{proof}

From now on, given \( f \in X^* \) and \( \alpha > 0 \), we consider the function \break \( p_{f,\alpha} : X \rightarrow [0, +\infty) \) defined by \(p_{f,\alpha}(x) := f(x) + \alpha \|x\|\), \(\forall x \in X\). Since \( f \) is linear and \( \|\cdot\| \) is a norm, it follows that \( p_{f,\alpha} \) is sublinear. The following lemma establishes some inequalities that describe the behavior of the level sets of \( p_{f,\alpha} \).

\begin{lemma} \label{lemadelmu_por_alfa}
Let  $X$ be a normed space, $f \in X^*$, $0<\alpha < \|f \|_{*} $, and $\rho>0$. The following statements hold.
\begin{itemize}
\item[(i)] For every  $0<\beta<\|f\|_{*}-\alpha$ and  $x \in X$ there exist $x_1,x_2 \in B_X(x,\rho)$ such that $p_{f,\alpha}(x_1)< p_{f,\alpha}(x)-\rho\beta   <p_{f,\alpha}(x)+\rho\beta<p_{f,\alpha}(x_2)$.
\item[(ii)] For every $x \in X$ we have the inequalities $-\|f \|_{*} \|x\| \leq p_{f,\alpha}(x) < 2 \|f \|_{*}\|x\| $.
\end{itemize}
\end{lemma}
\begin{proof}

\begin{itemize}
\item[(i)]  Consider $x'\in B_X$ such that $f(x')>\alpha+\beta$. By the linearity of \( f \), we have \( f(-\rho x') = -\rho f(x') \), and using the absolute homogeneity of the norm, \( \| -\rho x' \| = \rho \|x'\| \leq \rho \), since \( x' \in B_X \). Therefore,
\[
p_{f,\alpha}(-\rho x') = -\rho f(x') + \alpha \rho \|x'\| \leq -\rho f(x') + \alpha \rho.
\]
Now, since \( f(x') > \alpha + \beta \), it follows that
\[
-\rho f(x') + \alpha \rho < -\rho (\alpha + \beta) + \alpha \rho = -\rho \beta.
\]
Thus, we conclude that
\[
p_{f,\alpha}(-\rho x') < -\rho \beta.
\]
Using the subadditivity of \( p_{f,\alpha} \), we deduce that
\[
p_{f,\alpha}(x - \rho x') \leq p_{f,\alpha}(x) + p_{f,\alpha}(-\rho x') < p_{f,\alpha}(x) - \rho \beta,
\]
and since \( \|x - (x - \rho x')\| = \|\rho x'\| \leq \rho \), we have \( x - \rho x' \in B_X(x, \rho) \).

On the other hand, again by the sublinearity of $p_{f,\alpha}$, we have $p_{f,\alpha}(x+\rho x')\geq p_{f,\alpha}(x)-p_{f,\alpha}(-\rho x')> p_{f,\alpha}(x)+\rho \beta$ and $x+\rho x'\in B_X(x,\rho)$. 
\item[(ii)] A direct consequence of the definition of $p_{f,\alpha}$ and the dual norm.
\end{itemize}
\end{proof}
Let us recall that given a cone $C$, we define the $\delta$-conic neighborhood of $C$ as $C_{\delta}=\mbox{cone}(\{y \in X \colon d(y, C\cap S_X)\leq \delta\})$.  These conic neighborhoods play an important role  in the proof of Theorem ~\ref{teoolarioseparacioncoradiantes2} as interpolating cones.  Given a cone $C$, we denote by $C_{\geq \alpha}$ the co-radiant set associated with $C$ that is defined by $C_{\geq \alpha}:=\{x\in C\colon \|x\|\geq \alpha\}$.

The following lemma provides a co-radiant set associated to the conic neighbourhood \(C_{\delta}\) that contains certain balls centered at a co-radiant set associated to the initial cone \(C\).

\begin{lemma} \label{lemadelnu}
Let $C \subset X$ be a cone, take $0<\delta<1$, $d>0$, and $y \in C_{\geq d}$. Then  $ y+\delta d B_X \subset (C_{\delta})_{\geq (1-\delta)d}$.
\end{lemma}
\begin{proof}
As $y \in  C_{\geq d}$ it follows that $\frac{y}{\|y\|}+\delta B_X \subset C_{\delta} $. Since $C_{\delta}$ is a cone, multiplying by $\|y\|$, we get $y+ \delta \|y\| B_X \subset C_{\delta}\Rightarrow y+ \delta d B_X \subset C_{\delta}$ because $ \| y \|\geq d$. Furthermore, if $y' \in y+ \delta d B_X $, then 
$\|y' \|=\|y' -y+y\|\ \geq \|y\|- \|y'-y\| \geq d-\delta d=(1-\delta)d$, giving
$y' \in (C_{\delta})_{\geq (1-\delta)d}$.
\end{proof}

Now, we can prove our first separation theorem of co-radiant sets.
\begin{proof}[Proof of Theorem \ref{teoolarioseparacioncoradiantes2}]

(i)$\Rightarrow$(ii)  Assume that the pair $(\mbox{cone}(\cC),\mbox{cone}(\cK))$ satisfies the SSP and denote $K=\mbox{cone}(\cK)$ and $C=\mbox{cone}(\cC)$.  By Lemma \ref{lema:cono_ampliado_SSP}, there exists $0<\delta <1$ such that the pair $(C_{\delta},K)$ satisfies the SSP. Now, by Theorem \ref{Tma1_sep_conos_MMOR}, there exist \( \alpha_1, \alpha_2 \in \mathbb{R} \) with \( 0 < \alpha_1 < \alpha_2 \), and a functional \( f \in (C_{\delta})^{\#} \) such that \( (f, \alpha) \in (C_{\delta})^{a\#}_+ \) for every \( \alpha \in (\alpha_1, \alpha_2) \), and
\begin{equation}\label{eqsepara}
p_{f,\alpha}(y) < 0 < p_{f,\alpha}(y'),
\end{equation}
for every \( y \in -C_{\delta} \setminus \{ \cero \} \) and \( y' \in \mbox{bd}(-K) \setminus \{ \cero \} \). Furthermore, $f\in \cC^{\#}$ because $f\in (C_{\delta})^{\#} \subset  C^{\#}$ and $ \cC \subset C \setminus \{0_X\}$.  On the other hand, without loss of generality, we may assume that $f\in S_{X^*}$. Define $\lambda:=\inf\{f(y)- \alpha \| y \|: y \in \cC, \alpha \in (\alpha_1, \alpha_2) \}$, we will prove that $\lambda>0$ or, equivalently, that $-\lambda=\sup\{p_{f,\alpha}(y): y \in -\cC, \alpha \in (\alpha_1, \alpha_2) \}<0$. The inequality $-\lambda \leq 0$ is a direct consequence of (\ref{eqsepara}) because $-\cC \subset -C_{\delta} \setminus\{\cero\}$. Now assume, contrary to our claim, that $ \lambda=0$.   For every $\alpha \in (\alpha_1,\alpha_2)$, we have $(f,\alpha)\in  (C_{\delta})^{a\#}_+\subset C^{a\#}_+$,  which implies  $ \alpha<1$ as $f$ is taken in $S_{X^*}$. Then $\alpha_2 \leq 1$. However, without loss of generality, we may suppose that $\alpha_2<1$ because otherwise we can redefine $\alpha_2\in (\alpha_1,1)$. Now, consider $\mu=\frac{1-\alpha_2}{2}\delta d_{\cC}>0$. By the assumption $ \lambda=0$, there exist $ y_{0} \in -\cC \subset -C_{\geq d_{\cC}}$ and $\alpha_0 \in (\alpha_1, \alpha_2) $ such that 
\begin{equation} \label{eq12}
-\mu<p_{f,\alpha_0}(y_0)<0.
\end{equation}
Then, we apply Lemma \ref{lemadelmu_por_alfa} (i) for $\beta=\frac{1-\alpha_2}{2}$, $\rho=\delta d_{\cC}$, $\alpha=\alpha_2$ and $\|f\|_{*}$ and we get that there exists some $y_1 \in y_{0}+ \delta  d_{\cC}  B_X$ such that
\begin{equation}\label{eq22}
 p_{f,\alpha_0}(y_0)+\mu<p_{f,\alpha_0}(y_1).
\end{equation} 
Now, Lemma \ref{lemadelnu}  guarantees that $y_1\in -(C_{\delta})_{\geq (1-\delta)d_{\cC}}$. On the other hand, from (\ref{eq12}) and (\ref{eq22}) we get 
\begin{equation}\label{eq32}
0 =-\mu+\mu<p_{f,\alpha_0}(y_0)+\mu<p_{f,\alpha_0}(y_1).
\end{equation}
But $y_1  \in  -(C_{\delta})_{\geq (1-\delta)d_{\cC}}   \subset -C_{\delta}$, which contradicts (\ref{eqsepara}) if $y_1\not =0_X$. On the other hand, (\ref{eq32}) is impossible in case $y_1=0_X$. Therefore, we conclude that $\lambda>0$.  

We now proceed to show the second inequality in (\ref{separacioncoradiant2_bis}), i.e., the inequality $ f(y)- \alpha\|y\|>\lambda$,  for every $\alpha \in (\alpha_1, \alpha_2)$ and $y\in \overline{\mbox{co}}(\cC)$.  We now claim that $f(y)- \alpha \| y \|>\lambda$ for every $\alpha \in (\alpha_1, \alpha_2)$ and $y \in \cC$. Indeed, otherwise, there would exist $\bar{y} \in \cC$ and $\bar{\alpha} \in (\alpha_1, \alpha_2)$ such that $\lambda=f(\bar{y})- \bar{\alpha} \| \bar{y} \| $. Now, taking some  $\bar{\alpha}< \hat{\alpha}< \alpha_2$,  we get that $f(\bar{y})- \hat{\alpha} \| \bar{y} \|<\lambda $, contradicting the definition of $\lambda$.  Then, using the superlinearity of each funtcion $f(\cdot)-\alpha \parallel \cdot \parallel$ guarantees that $f(y)- \alpha \| y \|>\lambda$ for every $\alpha \in (\alpha_1, \alpha_2)$ and $y\in \mbox{co}(\cC)$.  Finally, consider $y \in \mbox{bd}(\mbox{co}(\cC))$.  Then, we pick a sequence $\{y_n \} \subset \mbox{co}(\cC)$ with $ \lim_n  y_n =y$. Suppose, by contradiction, that there exists  $\bar{\alpha} \in (\alpha_1, \alpha_2)$ such that $\lambda=f(y)- \bar{\alpha} \| y \| $.
We fix now some $ \alpha' \in (\bar{\alpha}, \alpha_2)$, and then we get $\lim_n(f(y_n)- \alpha ' \| y_n \|) =\lim_n( f(y_n)- \bar{\alpha}  \| y_n \|  - (\alpha'- \bar{\alpha} )  \| y_n \| )=f(y )-\bar{\alpha}  \| y \|-    (\alpha'- \bar{\alpha} )  \| y \| =\lambda -(\alpha'- \bar{\alpha} )  \| y \| <\lambda$. Consequently, there exists $n_0 \in \mathbb{N}$ such that   $ f(y_{n_0})- \alpha ' \| y_{n_0} \| <\lambda $ which is impossible because $y_{n_0}\in \mbox{co}(\cC)$.

Now, we fix $0<\eta <\frac{\lambda}{I_{\cK}}$. We will prove the first inequality in (\ref{separacioncoradiant2_bis}), or equivalently, the inequality
\begin{equation} \label{separacioncoradiant2}
-  \lambda<p_{f,\alpha}(y'),\,\forall \alpha_1<\alpha<\alpha_2,\,y' \in \mbox{bd}(-\cK (\eta)).
\end{equation}
To prove (\ref{separacioncoradiant2}),  we  fix $\xi>0$ such that $\eta=\frac{\lambda}{\xi+I_{\cK}}$ and define $\tau:=\lambda  \frac{\frac{\xi}{2}+I_{\cK}}{\xi+I_{\cK}}<\lambda  $.  By Proposition \ref{SSPepsiloncoradiant}, we have $I_{\cK(\eta)}=\eta I_{\cK}<(\frac{\xi}{2}+I_{\cK})\eta =  \tau$. Now pick an arbitrary   ${\bar{y}} \in  \mbox{bd}(-\cK(\eta))$. If $ \|{\bar{y}} \| \geq \tau$, then, by Lemma \ref{Lemma_antes_claims} (v), ${\bar{y}} \in\mbox{bd}(\mbox{cone}(-\cK(\eta)))= \mbox{bd}(\mbox{cone}(-\cK))=\mbox{bd}(-\cK)$. Now (\ref{eqsepara}) applies and we have $-  \lambda< 0< p_{f,\alpha}(\bar{y})$, for every $\alpha \in (\alpha_1, \alpha_2)$. To complete the proof, assume now that $ \| {{\bar{y}}} \| < \tau$.  Then Lemma \ref{lemadelmu_por_alfa} (iii) assures that $-\lambda    <-\tau  \leq p_{f,\alpha}(\bar{y})$,  for every $\alpha \in (\alpha_1, \alpha_2)$  finishing the proof of (\ref{separacioncoradiant2}).

(ii)$\Rightarrow$(i) By assumption, we consider $0<\alpha_1<\alpha_2$ and  $f\in S_{X^*}$  such that $\lambda:=\inf\{f(y)- \alpha \| y \|: y \in \cC, \alpha \in (\alpha_1, \alpha_2) \} > 0$ and, if $0<\eta<  \frac{\lambda }{I_{\cK}}$, then $f(y')- \alpha\|y'\|<   \lambda < f(y)- \alpha\|y\|$, for every $\alpha_1<\alpha<\alpha_2$, $y' \in \mbox{bd}(\cK(\eta))$, and $y \in\overline{\mbox{co}}(\cC)$.  Now, assume that $\cC$ admits a norm base. We fix $t>\max \{ I_{\cK(\eta)}, I_{\cC}\}$,  note that $t>0$ because $I_{\cK(\eta)}=\eta I_{\cK}>0$ by Proposition~\ref{SSPepsiloncoradiant}.  For any $y\in \cC\cap tS_X $ and $y' \in  (\mbox{bd}(\cK(\eta))\cap tS_X) \cup \{ 0_X \}$ we have that $f(y)>\lambda +\alpha t> f(y')$ for all $\alpha \in (\alpha_1,\alpha_2)$.   Now, fix $\alpha_3, \alpha_4 \in \mathbb{R} $ satisfying $\alpha_1 < \alpha_3 < \alpha_4 < \alpha_2$. Then,  we have $f(y)>\lambda+\alpha_4 t >\lambda+\alpha_3 t> f(y')$, for every $y\in \cC\cap tS_X $ and $y' \in  (\mbox{bd}(\cK(\eta))\cap tS_X) \cup \{ 0_X \}$, which also hold for $y\in \mbox{co}(\cC\cap tS_X)$ and $y' \in  \mbox{co}((\mbox{bd}(\cK(\eta))\cap tS_X) \cup \{ 0_X \})$. As a consequence, we have the inequality 
\begin{equation} \label{contradiccion}
f(y)-f(y')>(\alpha_4-\alpha_3)t>0,\,
\end{equation}
for every $y \in \mbox{co}(\cC \cap t S_X)$ and $y \in \mbox{co}(\mbox{bd}(\cK(\eta))\cap t S_X)\cup\{\cero\})$.
Now,  as $\mbox{cone}(\cK(\eta))= \mbox{cone}(\cK)$, it is sufficient to check that the pair  $(\mbox{cone}(\cC),\mbox{cone}(\cK(\eta)))$ satisfies the SSP.  Now,  contrary to our claim, we suppose that the pair $(\mbox{cone}(\cC),\mbox{cone}(\cK(\eta)))$ does not satisfy the SSP.  Then, by Proposition \ref{teoremadelosSSP} we have 
$$\cero  \in \overline{\mbox{co}(\cC \cap t S_X)-\mbox{co}((\mbox{bd}(\cK(\eta))\cap t S_X)\cup\{\cero\})}.$$ As a consequence, there exist two sequences $(y_n)_n \subset \mbox{co}(\cC \cap t S_X)$ and $(y'_n)_n \subset \mbox{co}((\mbox{bd}(\cK(\eta))\cap t S_X)\cup\{\cero\})$ such that $0=\lim_n \|y_n-y'_n \|$. Now, by continuity of $f$, we have $\lim_n f(y_n-y'_n)=0$. As a consequence, there exists $n_0 \in \mathbb{N}$ such that $  | f(y_{n_0})-f(y'_{n_0}) | < (\alpha_4-\alpha_3)t $,  which contradicts (\ref{contradiccion}). 
Then $(\mbox{cone}(\cC),\mbox{cone}(\cK(\eta)))$ satisfies the SSP and the proof is over.
\end{proof}

The fact that the cones generated by two co-radiant sets satisfy the SSP decisively influences the relative position of the co-radiant sets. The following result sheds light on this issue. 
\begin{cor} \label{posicion_relativa_de conos}
Let $X$ be a normed space, let $\cC$ and $\cK$ be co-radiant subsets of $X$ such that $\cK$ admits a norm base and $d_{\cC}>0$. If the pair $(\mbox{cone}(\cC),\mbox{cone}(\cK))$ satisfies the SSP, then there exists $0<\eta_{0} \leq 1$ such that given $0<\eta \leq \eta_{0}$ we have either $\overline{\mbox{co}}(\cC) \subset \mbox{int}(\cK(\eta))$ or $\overline{co}(\cC) \subset \mbox{int}(X \setminus \cK(\eta))$. 
\end{cor}

\begin{proof}
By Theorem \ref{teoolarioseparacioncoradiantes2}, there exist \( \alpha_1, \alpha_2 \in \mathbb{R} \) with \( 0 < \alpha_1 < \alpha_2 \), and \( f \in S_{X^*} \) such that \(\lambda := \inf \{ f(y) - \alpha \| y \| : y \in \cC,\, \alpha \in (\alpha_1, \alpha_2) \} > 0\), and for every \( 0 < \eta < \frac{\lambda}{I_{\cK}} \), the inequality
\begin{equation} \label{separacioncoradiant20}
f(y') - \alpha \| y' \| < \lambda < f(y) - \alpha \| y \|,
\end{equation}
holds for all \( \alpha \in (\alpha_1, \alpha_2) \), \( y' \in \mathrm{bd}(\cK(\eta)) \), and \( y \in \overline{\mathrm{co}}(\cC) \). Now, we  denote $\eta_{0}=\frac{ \lambda }{I_{\cK}}$.  Take $0<\eta <\eta_{0}$ and assume that there exist $\bar{c}, \bar{k} \in \overline{co}(\cC)$ such that  $\bar{c} \in \mbox{int}(\cK(\eta)) $  and   $\bar{k} \in \mbox{int}(X \setminus \cK(\eta))$. Then, there exists $\xi_0 \in (0,1)$ such that 
$x_0=\xi_0\bar{c}+(1-\xi_0)\bar{k} \in \mbox{bd}(\cK(\eta))$.  On the other hand,  from  (\ref{separacioncoradiant20}), we get that $f(\bar{c})- \alpha\|\bar{c}\| > \lambda  $ and $f(\bar{k})- \alpha\|\bar{k}\| >  \lambda$.  On the other hand, using the linearity of \( f \) and the triangle inequality, we obtain
\begin{align*}
f(x_0) - \alpha \|x_0\| &= f(\xi_0 \bar{c} + (1 - \xi_0) \bar{k}) - \alpha \|\xi_0 \bar{c} + (1 - \xi_0) \bar{k}\| \\
&\geq \xi_0 f(\bar{c}) + (1 - \xi_0) f(\bar{k}) - \alpha \xi_0 \|\bar{c}\| - \alpha (1 - \xi_0) \|\bar{k}\| \\
&= \xi_0(f(\bar{c}) - \alpha \|\bar{c}\|) + (1 - \xi_0)(f(\bar{k}) - \alpha \|\bar{k}\|) \\
&> \xi_0 \lambda + (1 - \xi_0) \lambda = \lambda,
\end{align*}
which contradicts (\ref{separacioncoradiant20}).
\end{proof}

Next, we state our second separation theorem. In this statement, there is one more assumption than in the previous one; specifically, it is assumed that the co-radiant sets fit into each other. This hypothesis allows for the separation of the set $\overline{\mbox{co}}(\cC)$ and the complementary of the set $\mbox{int}(K(\eta))$ by the same separating surface defined in Theorem \ref{teoolarioseparacioncoradiantes2}.  Regarding the proof, we only need to justify that (i)$\Rightarrow$(ii), but such a proof is the same as in the implication (i)$\Rightarrow$(ii) of Theorem 4.1, but using \cite[Theorem 3.3]{GarciaMelguizo2023} instead of Theorem \ref{Tma1_sep_conos_MMOR}.

\begin{theo} \label{thm:separacion_para_C_y_K_punto_comun}
Let $X$ be a normed space, let $\cC$ and $\cK$ be co-radiant subsets of $X$ such that $d_{\cC}d_{\cK}>0$,  $\cK$ admits a norm base,  and $\overline{co}(\cC) \cap \cK \neq \emptyset$.  Consider the following assertions.
\begin{itemize}
\item[(i)]  The pair $(\mbox{cone}(\cC),\mbox{cone}(\cK))$ satisfies the SSP.
\item[(ii)] There exist \( \alpha_1, \alpha_2 \in \mathbb{R} \) with  $0<\alpha_1<\alpha_2$ and  $f\in S_{X^*}$  such that: 
\begin{itemize}
\item[(a)] $\lambda:=\inf\{f(y)- \alpha \| y \|: y \in \cC, \alpha \in (\alpha_1, \alpha_2) \} > 0$,
%\item[(b)] $(f,\alpha)\in  \cC^{a\#}_{\lambda}$ for every $\alpha \in (\alpha_1, \alpha_2)$, 
\item[(b)] for every \( 0 < \eta < \frac{\lambda}{I_{\cK}} \), the inequalities
\[
f(y') - \alpha \| y' \| < \lambda < f(y) - \alpha \| y \|,
\]
hold for all \( \alpha \in (\alpha_1, \alpha_2) \), $y' \in X \setminus \mbox{int}(\cK(\eta))$, and \( y \in \overline{\mathrm{co}}(\cC) \).
\end{itemize}
\end{itemize}
Then (i)$\Rightarrow$(ii).  Moreover, if \( \cC \) also admits a norm base, then (ii)$\Rightarrow$(i).
\end{theo}

We conclude this section with our last  theorem of separation that is stated in the most general context possible as it does not require the co-radiant sets to admit a norm base.  One of the interests of this last result lies in the fact that many co-radiant sets do not admit a norm base, as illustrated in the following example.

\begin{example} \label{ejemplo}
The following examples illustrate co-radiant sets that do not admit a norm base:
\begin{itemize}
\item[(i)] $\cK_1:=\{ (x,y)\in \mathbb{R}^2: x>0, y\geq\frac{1}{x} \}$.
\item[(ii)] $\cK_2:=\{ (1+x,1+y)\in \mathbb{R}^2: x\geq 0, y\geq0 \}$. 
\item[(iii)] $\cK_3:=\{x\in \mathbb{R}^2: f(x)- \alpha \|x\|_2\geq \lambda\}$,  for any $f \in (\mathbb{R}^2)^*$, $0<\alpha<\|f\|_*$, and $\lambda>0$, being $\|\cdot\|_2$ the Euclidean norm on $\mathbb{R}^2$.
\end{itemize}
\end{example}
Despite being uncommon for a co-radiant set to admit a norm base, what does occur is that any co-radiant set can be expressed as a  union  of a monotonically non-decreasing sequence (in the sense of inclusion) of co-radiant subsets with a norm base.  Indeed, let $\cK \subset X $ be a co-radiant set.  Let us recall that, for every $r>0$, we define the set $\cK_r:=\cK\cap rS_X$. Now,  for every $n \in \mathbb{N}$ we define $\cK^n:= \cK \cap \mbox{cone}(\cK_n )$, if  $\cK_n\not = \emptyset$, and $\cK^n=\emptyset$, otherwise. The following lemma  is easy to check.
\begin{lemma}\label{lemma_previo_ultimo_Tma_separacion} The following properties hold.
\begin{itemize}
\item[(i)] $\cK^n\subset \cK^m$, whenever $n<m$.
\item[(ii)] $\cK= \cup^{\infty}_{n=1} \cK^n$.
\item[(iii)] If $\cK_n\not = \emptyset$,  then $\cK_n=\cK^n\cap nS_X$ and $\cK_n$  is a norm-base for $\cK^n$.
\item[(iv)] If $\cK$ admits a norm base, then  there exists $n_0 \in \mathbb{N}$ such that  $\cK= \cup^{n_0}_{n=1} \cK^n=\cK^{n_0}$. 
\end{itemize}
\end{lemma}

The preceding decomposition will be used in the proof of the next theorem, which shows that implication \( i) \Rightarrow ii) \) from Theorem~\ref{thm:separacion_para_C_y_K_punto_comun} remains valid for general co-radiant sets, provided that in the SSP assumption, the set \( \operatorname{cone}(\cK) \) is replaced by \( \operatorname{cone}(\cK_n) \) for some sufficiently large \( n \).

\begin{theo} \label{teoolarioseparacioncoradiantesAmpliado}
Let $X$ be a normed space, and let $\cC$ and $\cK$ be co-radiant subsets of $X$ such that   $d_{\cC}d_{\cK} >0$ and $\overline{co}(\cC) \cap \cK  \not = \varnothing$.  Assume that there exists  $n\in \mathbb{N}$ such that  the pair $(\mbox{cone}(\cC),\mbox{cone}(\cK_n))$ satisfies the SSP. Then, there exist $\alpha_1$, $\alpha_2\in \mathbb{R}$ with $0<\alpha_1<\alpha_2$ and $f\in S_{X^*}$  such that:  
\begin{itemize}
\item[(a)] $\lambda:=\inf\{f(y)- \alpha \| y \|: y \in \cC,\,\alpha\in(\alpha_1,\alpha_2)\}>0$, 
\item[(b)] for some $0<\eta\leq 1$, the inequalities
\begin{equation*} \label{separacioncoradiantAmpliado}
 f(y')- \alpha\|y'\|<\lambda < f(y)- \alpha\|y\|,
\end{equation*}
hold for all $\alpha\in (\alpha_1,\alpha_2)$, $y' \in X \setminus \mbox{int}(\cK(\eta))$, and $y \in \overline{co}(\cC)$.
\end{itemize}
\end{theo}

\begin{proof}
Since $\overline{\operatorname{co}}(\cC) \cap \cK \neq \varnothing$ and $\cK = \bigcup_{n=1}^{\infty} \cK^n$, there exists \( m \in \mathbb{N} \) such that \( \overline{\operatorname{co}}(\cC) \cap \cK^m \neq \varnothing \). By Corollary~\ref{lemma_previo_ultimo_Tma_separacion}~(i), we then conclude that \( \overline{\operatorname{co}}(\cC) \cap \cK^i \neq \varnothing \) for all \( i \geq m \). As a consequence, if the given \( n \) satisfies \( n \geq m \), we have \( \overline{\operatorname{co}}(\cC) \cap \cK^n \neq \varnothing \). On the other hand, if \( n < m \), Corollary~3.14 in \cite{GarciaMelguizo2024} ensures that the pair \( (\operatorname{cone}(\cC), \operatorname{cone}(\cK_m)) \) satisfies the SSP. Therefore, setting \( \ell := \max\{n, m\} \), we obtain that \( \overline{\operatorname{co}}(\cC) \cap \cK^{\ell} \neq \varnothing \) and the pair \( (\operatorname{cone}(\cC), \operatorname{cone}(\cK_{\ell})) \) satisfies the SSP. Fix this value of \( \ell \) and consider \( \cK^{\ell} \). Recall that \( \cK_{\ell} \) is a norm base of \( \cK^{\ell} \). Applying Theorem~\ref{teoolarioseparacioncoradiantes2} to \( \cC \) and \( \cK^{\ell} \), we obtain the existence of \( \alpha_1, \alpha_2 \in \mathbb{R} \) with \( 0 < \alpha_1 < \alpha_2 \), and \( f \in S_{X^*} \), such that
\[
\lambda := \inf \left\{ f(y) - \alpha \| y \| : y \in \cC,\ \alpha \in (\alpha_1, \alpha_2) \right\} > 0,
\]
and for every \( 0 < \eta < \frac{\lambda}{I_{\cK^{\ell}}} \), the inequalities
\begin{equation} \label{separacioncoradiantampliado1}
f(y') - \alpha \| y' \| < \lambda < f(y) - \alpha \| y \|,
\end{equation}
hold for all \( \alpha \in (\alpha_1, \alpha_2) \), \( y' \in \operatorname{bd}(\cK^{\ell}(\eta)) \), and \( y \in \overline{\operatorname{co}}(\cC) \). Moreover, the superlinearity and continuity of the map \( y \mapsto f(y) - \alpha \| y \| \) directly imply that
\[
\lambda = \inf \left\{ f(y) - \alpha \| y \| : y \in \overline{\operatorname{co}}(\cC),\ \alpha \in (\alpha_1, \alpha_2) \right\} > 0.
\]
To conclude, it remains to verify that \( f(y') - \alpha \| y' \| < \lambda \) for every \( y' \in X \setminus \operatorname{int}(\cK(\eta)) \). Without loss of generality, we may assume that \( \eta \leq 1 \). This yields the inclusion \( \cK^{\ell} \subset \cK^{\ell}(\eta) \), and thus there exists \( z \in \overline{\operatorname{co}}(\cC) \cap \cK^{\ell}(\eta) \). Then, by \eqref{separacioncoradiantampliado1}, we have \( f(z) - \alpha \| z \| > \lambda \), which implies that \( z \notin \operatorname{bd}(\cK^{\ell}(\eta)) \), and hence \( z \in \operatorname{int}(\cK^{\ell}(\eta)) \). Now assume, by contradiction, that there exists \( \bar{z} \in X \setminus \operatorname{int}(\cK(\eta)) \) such that \( f(\bar{z}) - \alpha \| \bar{z} \| \geq \lambda \). Since \( \cK^{\ell} \subset \cK \), it follows that \( \bar{z} \notin \operatorname{int}(\cK^{\ell}(\eta)) \). Moreover, by \eqref{separacioncoradiantampliado1}, we have \( \bar{z} \notin \operatorname{bd}(\cK^{\ell}(\eta)) \). Then, there exists \( \xi_0 \in (0,1) \) such that the convex combination \( x_0 = \xi_0 z + (1 - \xi_0) \bar{z} \in \operatorname{bd}(\cK^{\ell}(\eta)) \). Therefore,
$f(x_0)-\alpha \|x_0 \|= f(\xi_0 z+ (1-\xi_0) \bar{z}) -\alpha \Vert \xi_0 z+ (1-\xi_0) \bar{z} \Vert) \geq  \xi_0(f(z)-\alpha \Vert z \Vert) + (1- \xi_0)(f(\bar{z})- \alpha \Vert \bar{z} \Vert)  >  \xi_0\lambda + (1- \xi_0) \lambda  = \lambda $,
which contradicts \eqref{separacioncoradiantampliado1}, since \( x_0 \in \operatorname{bd}(\cK^{\ell}(\eta)) \).
\end{proof}

\section{Optimality Conditions for Approximate and Properly Approximate Solutions of Vector Optimization Problems} \label{Scalarization}

In this section, we derive necessary and sufficient conditions for approximate solutions via sublinear scalar optimization problems.  We begin establishing some notation and terminology in order to introduce the notion of $\varepsilon$-efficient point. Let $X$ be a normed space and $C \subset X$ a proper pointed convex cone. The cone $C$ induces a partial order $\leq$ on $X$ defined as follows: for $x, y \in X$, $y \leq x$ if and only if $x - y \in C$. In this context, given a non-empty set $A \subset X$, a point $x_0 \in A$ is said to be an efficient point of $A$, denoted as $x_0 \in \mbox{E}(A,C)$, if $A \cap (-C \setminus \{0_X\}) = \emptyset$. %Assuming that $C$ is solid, a point $x_0 \in A$ is said to be a weakly efficient point of $A$, denoted as $x_0 \in\mbox{WE}(A,C)$, if $A \cap (-\text{int}(C)) = \emptyset$.
Let introduce the concept of $\varepsilon$-efficiency by adapting \cite[Definition~3.2]{Gutierrez2006a} to our context. This notion is based on replacing the ordering cone with a proper pointed co-radiant set $\mathcal{C}$ that approximates it. Recall that we define $\mathcal{C}(\varepsilon) = \varepsilon \mathcal{C}$ for every $\varepsilon > 0$.

\begin{defi}  \label{def_casieficiente} % \cite{Dentcheva1996}
Let $X$ be a normed space,  let  $\cC \subset X$ be a proper pointed co-radiant subset,  let $A\subset X$, and $\varepsilon> 0$. We say that $x_0 \in A$ is an $\varepsilon$-efficient  point  with respect to $\cC $,  written $x_0 \in AE(A, \cC, \varepsilon)$, if $ (A-x_0)\cap(-\cC(\varepsilon) ) \subset \{\cero\}$.
\end{defi}

Next, we will  consider a particular type of co-radiant sets. Then, we will prove that its associated \(\varepsilon\)-efficient solutions are solutions to a scalar optimization problem, and afterwards,  we will analyze the connections between these solutions and the \(\varepsilon\)-efficient solutions generated by a co-radiant set included in it. 

Indeed, for any $f\in X^*$ and $\alpha\in \mathbb{R}$ with $0<\alpha<\|f\|_*$,  we define the open co-radiant set  $\cK_{f,\alpha}:=\{x\in X\colon f(x)-\alpha \|x\|>1\}$ which is said to be  a Bishop-Phelps co-radiant set.

\begin{prop}\label{remark_escalarizacion_gutierrez}
Let $X$ be a normed space, $f\in X^*$, and $0<\alpha<\|f\|_*$. Then the set $\cK_{f,\alpha}$ is a proper pointed co-radiant subset.
\end{prop}
\begin{proof}
Let us check that $\cK_{f,\alpha}\not = \emptyset$. Indeed, taking some $x_0\in S_X$ such that $f(x_0)>\alpha$, we get that $f(x_0)-\alpha \|x_0\|>0$.  Consequently, there exists $\mu>0$ such that $f(\mu x_0)-\alpha \|\mu x_0\|>1$, whence   $\mu x_0 \in \cK_{f,\alpha}$. Then, $\cK_{f,\alpha}$ is proper because $\cero \not \in \cK_{f,\alpha}$. The fact that $\cK_{f,\alpha}$ is a co-radiant subset is straightforward.  To check that $\cK$ is pointed we assume that $x\in \cK_{f,\alpha}\cap (-\cK_{f,\alpha})$. Then $f(x)+\alpha \|x\|<-1<1<f(x)-\alpha \|x\|$, that implies $\alpha<0$ (because $x\not =0_X$), a contradiction. 
\end{proof}

Here we recall that given $f\in X^*$ and $\alpha>0$, we define the  sublinear map $p_{f,\alpha}(x):=f(x)+\alpha \|x\|$,  $\forall x \in X$.
In the next proposition we show that the set of $\lambda$-efficient  points  with respect to the co-radiant set $\cK_{f,\alpha}$  can be obtained via   suitable scalarizations of the sublinear functional $p_{f, \alpha}$.

\begin{prop}\label{prop_GAMin=AE}
Let $X$ be a normed space, $f\in X^*$, $0<\alpha<\|f\|_*$, and $\lambda>0$. Then
$\mbox{AE}(A,\cK_{f,\alpha},\lambda)=\{x  \in A: p_{f, \alpha}(x'-x )\geq -\lambda \  \forall x' \in A\}$.
\end{prop}

\begin{proof}
Taking into account that   $-\cK_{f,\alpha}(\varepsilon)=\{x \in X\colon p_{f,\alpha}(x)<-\lambda\}$ the following chain of equivalences becomes immediate. Indeed,   
$x_0\in \mbox{AE}(A,K_{f,\alpha},\lambda)$
$\Leftrightarrow$
$(A-x_0)\cap (-\cK_{f,\alpha}(\lambda))=\varnothing$
$\Leftrightarrow$
$(A-x_0)\subset \{x \in X\colon p_{f,\alpha}(x'-x_0)\geq -\lambda\}$
$\Leftrightarrow$
$p_{f, \alpha}(x'-x_0 )\geq -\lambda \  \forall x' \in A.$

\end{proof}

\begin{remark}
In view of the previous result, it is worth noting that the points of \( \mbox{AE}(A, K_{f,\alpha}, \lambda) \) are solutions to scalar optimization problems. Therefore, the following results in this section provide conditions in terms of scalarization, as this set is used in their statements. Sometimes, necessary conditions are provided, and other times, sufficient conditions are given.
\end{remark}

The following result provides sufficient conditions for $\varepsilon$-efficient points in terms of scalarization.  Here arises the natural question whether or not every $\varepsilon$-efficient  point can be reached this way. This question remains open for future research.

\begin{theo} \label{theorem_complementaria_escalarizacion}
Let $X$ be a normed space, let $\cC \subset X$ be a co-radiant subset with $d_{\cC}>0 $, let $A\subset X$, $\varepsilon>0$, and $\lambda>0$. The following statements hold.
\begin{itemize}
%\item[(i)]If $(f,\alpha) \in \cC^{a\#}_{\lambda}(\varepsilon)$, then   $\text{AMin}(p_{f,\alpha}, A,\lambda) \subseteq \mbox{AE}(A,\cK_{f,\alpha},\lambda) \subseteq AE(A, \cC, \varepsilon)$.
\item[(i)]If $(f,\alpha) \in \cC^{a\#}_{\lambda}(\varepsilon)$, then   $\mbox{AE}(A,\cK_{f,\alpha},\lambda) \subseteq AE(A, \cC, \varepsilon)$.
\item[(ii)] If $(f,\alpha) \in \cC^{a*}_{\lambda}(\varepsilon)$ and $x_0 \in \mbox{AE}(A,\cK_{f,\alpha},\lambda)$,  
then $$\left( \left[ x_0+\{x\in X\colon p_{f,\alpha}(x)<0\}\right] \cap A \right)  \subset AE(A, \cC, \varepsilon).$$
\end{itemize}
\end{theo}

\begin{proof}

(i) The homogeneity of $f(\cdot)-\alpha \|\cdot\| $  and  $(f,\alpha) \in \cC^{a\#}_{\lambda}(\varepsilon)$  yield $ \cC(\varepsilon)\subset  \cK_{f,\alpha}(\lambda)$, and the inclusion becomes immediate.

(ii) We first consider the case $x_0=0_X$.  
Pick any $z\in \{x\in X\colon p_{f,\alpha}(x)<0\}\cap \cA$, we will check that $A\cap (z-\cC(\varepsilon))\subset \{z\}$.
Assume, contrary to our claim, that there exists $y\in A\cap (z-\cC(\varepsilon))$, $y\not = z$.
Since $y-z\in -\cC(\varepsilon)$ and the function $p_{f,\alpha}$ is sublinear, it follows that $p_{f,\alpha}(y)-p_{f,\alpha}(z)\leq p_{f,\alpha}(y-z)\leq -\lambda$.
On the other hand, as $0_X\in \mbox{AE}(A-x_0, \cK_{f,\alpha},\lambda)$, by Proposition \ref{prop_GAMin=AE} we obtain that $p_{f,\alpha}(y)\geq -\lambda$. Since, in addition, it holds that $p_{f,\alpha}(z)<0$, we obtain the inequality $p_{f,\alpha}(y)-p_{f,\alpha}(z)\geq -\lambda-p_{f,\alpha}(z)>-\lambda$ which contradicts the previous one. 
Now, consider the case $x_0\not =0_X$. 
Since $x_0 \in \mbox{AE}(A,\cK_{f,\alpha},\lambda)$, by Proposition \ref{prop_GAMin=AE} we get that  $ p_{f,\alpha}(x) \geq -\lambda \ \forall x \in A-x_0$.  
Now, by the previous case, we have $ \{x\in X\colon p_{f,\alpha}(x)<0\} \cap  \left(A-x_0 \right)  \subset \mbox{AE}(A-x_0, \cC, \varepsilon).$
By translation, taking into account the sublinearity of $p_{f,\alpha}$ we get the inclusion $\left( \left[ x_0+\{x\in X\colon p_{f,\alpha}(x)<0\}\right] \cap A \right)  \subset \mbox{AE}(A, \cC, \varepsilon)$.

\end{proof}

The following definition was introduced for the first time in \cite{Sayadibander2017}. Let us recall that a set is said to be solid if it has  non-empty interior.
\begin{defi}
Given \( \varepsilon \geq 0 \), a subset \( A \subset X \), and a proper pointed co-radiant subset \( \cC \subset X \), we say that a point \( x_0 \in A \) is an \( \varepsilon \)-properly efficient point with respect to \( \cC \), written \( x_0 \in PAE(A, \cC, \varepsilon) \), if there exists a proper solid pointed co-radiant subset \( \cK\subset X \) such that \( \cC \subset \mathrm{int}(\cK) \) and \( x_0 \in AE(A, \cK, \varepsilon) \).
\end{defi}

\begin{prop} \label{prop_proper_eficiencia}
Let $X$ be a normed space, let $\cC \subset X$ be a co-radiant subset  such that $d_{\cC} >0$,  let $A\subset X$, $\varepsilon>0$, and $\lambda>0$.  If $(f,\alpha) \in \cC^{a\#}_{\lambda}(\varepsilon)$, then $ AE(A, \cK_{f,\alpha}, \lambda) \subset \mbox{PAE}(A, \cC, \varepsilon)$.
\end{prop}
\begin{proof}
Consider $(f,\alpha)\in \cC^{a\#}_{\lambda}(\varepsilon)$ and define $\cK:=\{x \in X\colon f(x)-\alpha \|x\|>\frac{\lambda}{\varepsilon}\}$.  It is clear that $\cK$ is a proper solid pointed co-radiant set containing $\cC$. Now fix an arbitrary $x_0\in AE(A, \cK_{f,\alpha}, \lambda)$, we will check that $x_0\in \mbox{AE}(A,\cK,\varepsilon)$ by showing that $(A-x_0)\cap (-\cK(\varepsilon))\subset \{0_X\}$.
If $y\in A-x_0$,  since $x_0\in AE(A, \cK_{f,\alpha}, \lambda)$, it follows that $p_{f,\alpha}(y)=f(y)+\alpha \|y\|\geq -\lambda$. On the other hand, since $y \in -\cK(\varepsilon)$, it follows that $y\not= \cero$ and $f(y)+\alpha\|y\|<-\lambda$, a contradiction.  As a consequence, $x_0\in \mbox{PAE}(A,\cC,\varepsilon)$.
\end{proof}

In Theorem \ref{theorem_complementaria_escalarizacion} and Proposition \ref{prop_proper_eficiencia}, the assumption that the augmented dual cones are non-empty is crucial for the results to hold. The following result shows that those augmented dual cones are non-empty if the cone generated by the co-radiant set \( \cC \) has a bounded base.

\begin{prop}\label{prop_augmented_dual_cone_nonempty}
Let $X$ be a normed space and let $\cC\subset X$ be a co-radiant subset  such that $d_{\cC} >0$. If $\mbox{cone}(\cC)$ has a bounded base, then there exists $\lambda>0$ such that
\( \cC^{a*}_{\lambda}(\varepsilon)\not = \emptyset \) and \( \cC^{a\#}_{\frac{\lambda}{2}}(\varepsilon)\not = \emptyset \), $\forall \varepsilon>0$.  
\end{prop} 
\begin{proof}
Let $\cC$ be a co-radiant set such that $d_{\cC}>0$ and $C:=\mbox{cone}(\cC)$ has a bounded base.  By \cite[Theorem 1.1]{GARCIACASTANO20151178},  there exists $f\in S_{X^*}$ such that $\gamma:=\inf_{C\cap S_X}f>0$.  Then $f(\frac{c}{\|c\|})\geq \gamma$, $\forall c \in \cC$ $\Rightarrow$ $f(c)-\gamma \|c\|\geq 0$, $\forall c \in \cC$.  Now, fix arbitrary $\varepsilon>0$ and $x\in \cC(\varepsilon)$. Then $f(x)-\frac{\gamma}{2} \|x\|=f(x)-\gamma \|x\|+\frac{\gamma}{2} \|x\|\geq \frac{\gamma}{2} \|x\|\geq \frac{\gamma \varepsilon d_{\cC}}{2}$.
Then, taking $\lambda:=\frac{\gamma \varepsilon d_{\cC}}{2}$ we have that $(f,\frac{\gamma}{2}) \in \cC^{a*}_{\lambda}(\varepsilon)$ and $(f,\frac{\gamma}{2}) \in \cC^{a\#}_{\frac{\lambda}{2}}(\varepsilon)$.
\end{proof}

\begin{defi}
Let $X$ be a normed space and let $\cC\subset X$ be a co-radiant subset such that $d_{\cC} >0$ and consider $0<\delta<d_{\cC}$. We define the $\delta$-co-radiant neighborhood of $\cC$ by $$\cC_{\delta}:=\mbox{co-rad}(\cC+\delta B_X).$$
\end{defi}
It is clear that $AE(A,\cC_{\delta},\varepsilon)\subset PAE(A,\cC,\varepsilon)$. Furthermore, the following result suggests that the set $AE(A,\cC_{\delta},\varepsilon)$ provides a good approximation of $AE(A,\cC,\varepsilon)$  for small values of $\delta>0$.

\begin{lemma}
Let $X$ be a normed space and let $\cC\subset X$ be a co-radiant subset such that $d_{\cC} >0$.  The following assertions hold.
\begin{itemize}
\item[(i)] $0\not \in \overline{\cC_{\delta}}$, for every $0<\delta
< d_{\cC}$.
\item[(ii)]  If $0<\delta_1<\delta_2<d_{\cC}$, then $\cC_{\delta_1}\subset \cC_{\delta_2}$.
\item[(iii)] $\cap_{0<\delta
< d_{\cC}}\cC_{\delta} =\overline{\cC}$.
\end{itemize}
\end{lemma}

\begin{proof}
It is only necessary to check  (iii).  To show $\subset$ we fix $n_0 \geq 1$ such that $\frac{1}{n_0}< d_C$ and pick an arbitrary $x\in \cap _{n > n_0}\cC_{\frac{1}{n}}$.  For every $n \geq n_0$ we  write $x=\lambda_n x_n=\lambda_n(c_n+u_n)$, for $\lambda_n \geq 1$,  $c_n \in \cC$, and $u_n\in B(0,\frac{1}{n})$.  Then $\frac{\|x\|}{\lambda_n}=\|x_n\|=\|u_n+c_n\|\geq \|c_n\|-\|u_n\|>d_{\cC}-\frac{1}{n_0}>0$. Then the sequence $(\lambda_n)_n$ is bounded and, as a consequence, we can assume that it converges to some $\lambda\geq 1$.  Now, we have $\|\frac{x}{\lambda}-c_n\|\leq |\frac{x}{\lambda}-\frac{x}{\lambda_n}\|+|\frac{x}{\lambda_n}-c_n\|=|\frac{x}{\lambda}-\frac{x}{\lambda_n}\|+\|u_n\|$, for every $n \geq n_0$. It follows that $\frac{x}{\lambda} \in \overline{\cC}$. Finally, since $\cC$ is a co-radiant set, so is $\overline{\cC}$. Then $x \in \overline{\cC}$. The inclusion $\supset$ is a direct consequence of the known equality $\overline{\cC}=\cap_{n\geq 1}(\cC + \frac{1}{n}B_X)$.
\end{proof}

In the following result, we obtain, under the assumption of the SSP, a necessary condition for $\varepsilon$-properly efficiency in terms of scalarization.  Since  $\mbox{PAE}(A, \cC, \varepsilon)\subset \mbox{AE}(A, \cC, \varepsilon)$, this result is related to the question arisen before Theorem \ref{theorem_complementaria_escalarizacion}.

\begin{theo}\label{coro_1:escalarizacion22}
Let $X$ be a normed space and let \( \cC \subset X \) be a co-radiant subset with \( d_{\cC} > 0 \), and let \( A \subset X \). Assume that there exist \( \delta \in \mathbb{R} \) with \( 0 < \delta < d_{\cC} \), and \( n \in \mathbb{N} \) such that \( \cC \cap n S_X \neq \emptyset \), and the pair \( (\mathrm{cone}(\cC), \mathrm{cone}(\cC_{\delta} \cap n S_X)) \) satisfies the SSP. Then, there exist \( \alpha_1, \alpha_2 \in \mathbb{R} \) with \( 0 < \alpha_1 < \alpha_2 \leq 1 \), \( f \in S_{X^*} \), and \( 0 < \eta \leq 1 \) such that:
\begin{itemize}
\item[(i)] $\lambda := \inf \left\{ f(x') - \alpha \| x' \| : x' \in \cC,\, \alpha \in (\alpha_1, \alpha_2) \right\} > 0$.
\item[(ii)] 
If \( 0 < \varepsilon < 1 \), \( \alpha \in (\alpha_1, \alpha_2) \), and \( x_0 \in AE(A, \cC_{\delta}(\eta), \varepsilon) \), then:
\begin{equation} \label{rdo1}
p_{f,\alpha}(x - x_0) > -\lambda \varepsilon, \quad \forall x \in A,
\end{equation}
and
\begin{equation*} \label{rdo2}
\left[ x_0 + \left\{ x \in X : p_{f,\alpha}(x) < 0 \right\} \right] \cap A \subset AE(A, \cC, \varepsilon).
\end{equation*}
\end{itemize}
\end{theo}

\begin{proof}
The condition $\cC\cap nS_X\not = \emptyset$ implies that $\overline{\mbox{co}}(\cC)\cap \cC_{\delta}\cap \mbox{cone}(\cC_{\delta}\cap nS_X)\not =\emptyset$.  Then, we apply Theorem  \ref{teoolarioseparacioncoradiantesAmpliado} and we have $f \in S_{X^*}$,  $0<\alpha_1<\alpha_2$ such that $\lambda:=\inf\{f(x')- \alpha \| x' \|: x' \in \cC,\,\alpha\in(\alpha_1,\alpha_2)\}>0$,  and there exists $0<\eta\leq 1$ such that for  $ f(y)- \alpha\|y\| <\lambda  <  f(x')- \alpha\|x'\|$,  for every $\alpha_1<\alpha<\alpha_2$, $y \in X \setminus \mbox{int}(\cC_{\delta}(\eta))$ and $x' \in \overline{\mbox{co}}(\cC)$.  To check $\alpha_2\leq 1$, consider arbitrary $\alpha_1<\alpha<\alpha_2$ and $x' \in \cC$. Then $f(x')>\alpha \|x'\|$, which implies that $f(\frac{x'}{\|x'\|})>\alpha$, $\forall x' \in \cC$.  As a consequence,  $f\in  \cC^{\#}$ and $1>\alpha$.  As $\alpha_1<\alpha<\alpha_2$ was arbitrarily taken, the inequality is proved.   In particular,  we also have that $(f,\alpha)\in \cC^{a\#}_{\lambda \varepsilon}(\varepsilon)$, $\forall \varepsilon>0$. On the other hand, it is clear that $p_{f,\alpha}(x')< - \lambda  < p_{f,\alpha}(y)$, for every $\alpha_1<\alpha<\alpha_2$, $y \in Y \setminus \mbox{int}(-\cC_{\delta}(\eta))$ and $x' \in \overline{\mbox{co}}(-\cC)$. Let us fix some $0<\varepsilon<1$. Then, $p_{f,\alpha}(y)> -\lambda \varepsilon$, for every $y \in Y \setminus \mbox{int}(-\cC_{\delta}(\eta\varepsilon))$ and $\alpha_1<\alpha<\alpha_2$.  Finally, since $x_0 \in \mbox{AE}(A, \cC_{\delta}(\eta), \varepsilon)$, we have the inclusion $A \subset X\setminus (x_0-\cC_{\delta}(\eta\varepsilon))$, which implies that $p_{f,\alpha}(x-x_0) > -\lambda \varepsilon$, for $x \in A$ 
proving (\ref{rdo1}).  Finally, since $(f,\alpha)\in \cC^{a*}_{\lambda\varepsilon}(\varepsilon)$, we can apply Theorem \ref{theorem_complementaria_escalarizacion} and get
\begin{equation*}
\left( \left[ x_0+\{x\in X\colon p_{f,\alpha}(x)<0\}\right] \cap A) \right)  \subset AE(A, \cC, \varepsilon).
\end{equation*}
\end{proof}

\begin{remark}
Theorem \ref{coro_1:escalarizacion22} is a version of \cite[Theorem 6]{Sayadibander2017} for infinite-dimensional spaces, in which we have removed the requirements for the co-radiant subset $\cC$ to be convex, closed, and having a compact base. 
\end{remark}

\begin{remark}
Proposition \ref{prop_GAMin=AE} and Theorem \ref{theorem_complementaria_escalarizacion} provide a sublinear scalarization method and a sufficient condition for $\varepsilon$-efficiency.  In
Proposition \ref{prop_proper_eficiencia} such a scalarization gets sufficient conditions for $\varepsilon$-properly efficiency.  In Theorem  \ref{theorem_complementaria_escalarizacion}  and Proposition  \ref{prop_proper_eficiencia} the method relies on the assumption that the augmented dual cones are non-empty, which is guaranteed by Proposition~\ref{prop_augmented_dual_cone_nonempty} when the cone generated by the co-radiant subset has a bounded base. Finally, in Theorem \ref{coro_1:escalarizacion22} the sublinear scalarization provides, under the assumption of the SSP, a necessary condition for $\varepsilon$-properly efficiency. 
In our results, no additional assumptions on the co-radiant subset \(\mathcal{C}\) or convexity conditions on the set \(A\) are required.
\end{remark}

Let us introduce another notion to be used in our last remark.  Fix $g: X \rightarrow \mathbb{R}$, $A \subset X$,  and $A \not = \emptyset$. For every  $\varepsilon > 0$,  it is defined the set of $\varepsilon$-approximate solutions of $\mbox{Min} \{g(x) \colon x \in A\}$ as $\text{AMin}(g, A,\varepsilon) := \{x_0 \in A \colon g(x_0) - \varepsilon \leq g(x), \, \forall x \in A\}.$ It is easy to check that $\cap_{\varepsilon>0}\text{AMin}(g, A,\varepsilon)=\mbox{Min} \{g(x) \colon x \in A\}$.   The following formula 
$$\text{AMin}(p_{f,\alpha}, A,\varepsilon)\subsetneq \{x_0\in A\colon 0_X\in \mbox{AMin}(p_{f,\alpha},A-x_0,\varepsilon)\}=\mbox{AE}(A,\cK_{f,\alpha},\varepsilon),$$
shows that the set of $\varepsilon$-approximate solutions considered in \cite{Gutierrez2006a} is a strict subset of the set of $\varepsilon$-efficient points we have handled in this section when $g=p_{f,\alpha}$. 

\begin{remark}\label{remark_escalarizacion_gutierrez_2}
In \cite{Gutierrez2006a}, some necessary  and sufficient conditions for $\varepsilon$-efficiency were obtained via $\varepsilon$-approximate solutions of related scalar optimization problems.  For necessary conditions, the authors had to assume that $\cC$ is star-shaped with kern$(\cC)$ solid  and, for sufficient conditions,  they had to use an extra strictly local monotone function.  In our approach, we find it more convenient to consider suitable functions $g$ that are sublinear; in particular, the role of $g$ here is played by $p_{f,\alpha}$. This choice allows us not only  to avoid the restrictions on $\cC$, but  consider a bigger set of solutions for the scalarizations.
\end{remark}

\section*{Acknowledgements.} The authors thank the anonymous referee for their careful reading and insightful suggestions, which helped improve the quality and presentation of the manuscript.

\section*{Funding}
The authors acknowledge the financial support from Project PID2021-122126NB-C32 funded by MICIU/AEI /10.13039/501100011033/ and FEDER A way of making Europe.

%\bibliographystyle{abbrv}
%
%\bibliography{ReferenciasCoradiant}

\end{document}